\documentclass[a4paper, 11pt]{article}

\usepackage{amsmath}
\usepackage{amsfonts}
\usepackage{amssymb}
\usepackage[english]{babel}
\usepackage{amsthm}
\usepackage{pgf,tikz}
\usepackage{amstext} 

\newtheorem{theorem}{Theorem}[section]
\newtheorem{corollary}[theorem]{Corollary}
\newtheorem{lemma}[theorem]{Lemma}

\theoremstyle{definition}
\newtheorem{definition}[theorem]{Definition}
\theoremstyle{definition}

\theoremstyle{definition}
\newtheorem{example}[theorem]{Example}

\newcommand{\Z}{\mathbb{Z}}

\newcommand{\F}{\mathbb{F}}

\newcommand{\comment}[1]{}
\numberwithin{equation}{section}

\begin{document}
	\date{}
	\title{$\mathrm{ EA}(q)$-additive Steiner 2-designs}
	\author{Marco Buratti \thanks{Sapienza Universit\`a di Roma,  Dipartimento di Scienze di Base e Applicate per l'Ingegneria (SBAI), Via Antonio Scarpa 14, 00161 Roma, Italy, email: marco.buratti@uniroma1.it} \and
	Mario Galici \thanks{Dipartimento di Matematica e Fisica ``Ennio de Giorgi", Universit\`a del Salento, Lecce, Italy, email: mario.galici@unisalento.it}  \and
	Alessandro Montinaro
	\thanks{Dipartimento di Matematica e Fisica ``Ennio de Giorgi", Universit\`a del Salento, Lecce, Italy,  email: alessandro.montinaro@unisalento.it}  \and
	Anamari Naki\'c\thanks{University of Zagreb, Faculty of Electrical Engineering and Computing, Unska 3, 10000 Zagreb, Croatia, email: anamari.nakic@fer.unizg.hr}  \and
	Alfred Wassermann\thanks{Department of Mathematics, University of Bayreuth, D-95440 Bayreuth,
	Germany, email: alfred.wassermann@uni-bayreuth.de}}

	\maketitle

\begin{abstract}
A design is $G$-additive with $G$ an abelian group, if its points are in $G$ and each block is zero-sum in $G$. 
All the few known ``manageable" additive Steiner 2-designs are $\mathrm{EA}(q)$-additive for a suitable $q$, where $\mathrm{EA}(q)$ is the elementary abelian group of order $q$. We present some general constructions for $\mathrm{EA}(q)$-additive Steiner 2-designs which unify the known ones and allow to find a few new ones: an additive $\mathrm{EA}(2^8)$-additive 2-$(52,4,1)$ design which is also resolvable, and three pairwise non-isomorphic
$\mathrm{EA}(3^5)$-additive 2-$(121,4,1)$ designs, none of which is the point-line design of $\mathrm{PG}(4,3)$.
In the attempt to find also an $\mathrm{EA}(2^9)$-additive 2-$(511,7,1)$ design, we prove that a putative
2-analog of a 2-$(9,3,1)$ design cannot be cyclic.
\end{abstract}
\noindent \small{\bf Keywords:} \scriptsize
additive design; Steiner 2-design; cyclic design; 1-rotational design; $q$-analog of a classic design; Segre variety.

\normalsize
\eject

\section{Introduction}
A 2-$(v,k,\lambda)$ {\it design} is a pair $(V,{\cal B})$ where $V$ is a $v$-set whose elements are called {\it points}
and $\cal B$ is a collection of $k$-subsets of $V$ called {\it blocks} such that every pair of distinct points is contained
in precisely $\lambda$ blocks. To avoid trivial designs, it is assumed that $v>k>2$.
In the special case that $\lambda$ is 1 the design is said to be a {\it Steiner $2$-design}.
Throughout the paper, we will write ``$(v,k,\lambda)$-design" rather than ``2-$(v,k,\lambda)$ design".
An {\it automorphism} of a design $(V,{\cal B})$ is a permutation on $V$ leaving $\cal B$ invariant.
A design is said to be {\it cyclic} if it admits an automorphism cyclically permuting all its points, whereas it
is said to be {\it $1$-rotational} if it admits an automorphism fixing one point and cyclically permuting all the others.
Finally, a $(v,k,\lambda)$-design $(V,{\cal B})$ is said to be {\it resolvable} if there exists a partition 
of $\cal B$ ({\it resolution}) into classes ({\it parallel classes}) each of which is a partition of $V$.

The nice notion of an additive design has been introduced in \cite{CFP}.
A $(v,k,\lambda)$-design is said to be {\it additive} if, up to isomorphism, its
points are elements of an abelian group $G$ and its blocks are all zero-sum in $G$.
As in \cite{BN1}, a design as above will be said $G$-additive.

The characteristics of additive designs make them highly elegant combinatorial objects. It is worth noting that the 
use of zero-sum blocks in the construction of combinatorial designs is a common occurrence 
(see, e.g., \cite{BCHW, BH, BP, C, EW, K, PD, WW}). In addition, the motivation to examine and construct additive designs comes from their significant connections with various areas of discrete mathematics, including coding theory and additive combinatorics.


Throughout the paper $\mathrm{EA}(q)$ will denote the elementary abelian group of order $q$
which, of course, can be seen as the additive group of $\F_q$, the field of order $q$.
We denote by $\mathrm{AG}(n,q)$ and $\mathrm{PG}(n,q)$ the $n$-dimensional affine and, respectively, projective geometry 
over $\F_q$.
A {\it $q$-analog of a $(v,k,\lambda)$-design}
is a $({q^v-1\over q-1},{q^k-1\over q-1},\lambda)$-design
with point-set $V=PG(v-1,q)$ whose blocks are $(k-1)$-dimensional subspaces of $V$.
It has been proved in \cite{BN2} that every such design is $\mathrm{EA}(q^v)$-additive.

Additive designs with $\lambda>1$ are not so rare (see, e.g., \cite{BN0,BMN,CF,CFP2,N,P,T}).
On the other hand, the additive $(v,k,1)$-designs appear to be ``precious".
Indeed, the few known results on additive Steiner 2-designs can be summarized as follows.
\begin{theorem}\label{summary} The following facts are known.
\begin{enumerate}
\item[(i)] The point-line design associated with $\mathrm{AG}(n,q)$ is $\mathrm{EA}(q^n)$-additive;
\item[(ii)] the point-line design associated with $\mathrm{PG}(n,q)$ is $\mathrm{EA}(q^{n+1})$-additive (see \cite{CFP,BN2});
\item[(iii)] the well celebrated $2$-analogs of a $(13,3,1)$ design \cite{BEOVW} are
$\mathrm{EA}(2^{13})$-additive $(8191,7,1)$-designs;
\item[(iv)] there exists an $\mathrm{EA}(5^3)$-additive $(124,4,1)$-design (see \cite{BMN}).
\item[(v)] if $k$ is neither singly even nor of the form $2^n3$, then there exists a $G\times EA(q)$-additive $(kq,k,1)$-design 
for a suitable zero-sum group $G$ of order $k$ and a suitable power $q$ of a prime divisor of $k$  (see \cite{BN1}).
\end{enumerate}
\end{theorem}

Result $(v)$ is important from a theoretical point of view since, for the time being, it is
the only one assuring the existence of an additive $(v,k,1)$-design for some
values of $v$ when $k$ is not of the forbidden form and neither a prime power nor a prime power plus one. On the other hand, it is not practical at
all since it leads to very cumbersome designs. Indeed, the mentioned ``suitable" power $q$ has 
to be really huge in comparison with $k$.
Thus, the known ``manageable"  additive Steiner 2-designs are all $\mathrm{EA}(q)$-additive
for a suitable prime power $q$.
Inspired by this, here we want to investigate further the $\mathrm{EA}(q)$-additive $(v,k,1)$-designs.
In the next section, we first establish which are the admissible prime powers $q$
and then we give two theoretical constructions for $\mathrm{EA}(q)$-additive $(v,k,1)$-designs
which are either cyclic with point-set the $v$-th roots of unity of $\F_q$ with $q\equiv1$ (mod $v$) admissible, 
or 1-rotational with point-set the $(v-1)$-th roots of unity of $\F_q$ and 0 with $q\equiv1$ (mod $v-1$) admissible.
These constructions are put in practice in the subsequent two sections where
we find a 1-rotational $\mathrm{EA}(2^8)$-additive $(52,4,1)$-design which is also resolvable, and three pairwise non-isomorphic
cyclic $\mathrm{EA}(3^5)$-additive $(121,4,1)$-designs,  none of which is the point-line design of $\mathrm{PG}(4,3)$.
In the last section, we discuss the possible existence of cyclic $\mathrm{EA}(2^n)$-additive $(2^n-1,k,1)$-designs.
At the moment they are known to exist only when $k=3$ (they come from Theorem \ref{summary} $(ii)$) and when 
$(n,k)=(13,7)$ in view of Theorem \ref{summary} $(iii)$.
The least possible unknown case is $(n,k)=(9,7)$: does there exist a cyclic $\mathrm{EA}(2^9)$-additive $(511,7,1)$-design? 
Unfortunately, we are still not able to answer this question. If it exists, however, it cannot be the 2-analog of a
$(9,3,1)$-design. We prove this by means of two different approaches: the Kramer-Mesner method and a geometrical method.

\section{Some general results}

Here we give some general results on the $\mathrm{EA}(q)$-additive $(v,k,1)$-designs
starting with a useful restrictive condition on the possible prime powers $q$. 

\begin{theorem}\label{admissible}
If there exists an $\mathrm{EA}(q)$-additive $(v,k,1)$ design, then $q$ is a power of a
prime  divisor of ${v-k\over k-1}$.
\end{theorem}
\begin{proof}
Recall that the number of blocks through any point of a $(v,k,1)$-design is the constant $r:={v-1\over k-1}$
known as the {\it replication number}.
Let $x$ be any point and let $B_1$, \dots, $B_r$ be the blocks through $x$.
Obviously, the singleton $\{x\}$ and the $r$ sets $B_1\setminus\{x\}$, \dots, $B_r\setminus\{x\}$ form a partition of the point set.
Considering that each $B_i$ is zero-sum, the sum of all elements of $B_i\setminus\{x\}$ is equal to $-x$ for every $i$. Then the
sum $\sigma$ of all the points is equal to $x+r(-x)=(1-r)x$. Thus we have $\sigma=(1-r)x$ for every point $x$. 
It follows that $(1-r)(x-y)=0$ for any pair of points $x$ and $y$, hence the order of any difference of two points
is a divisor of $r-1={v-k\over k-1}$. The assertion follows considering that the order of any non-zero element of $\F_q$ 
is the characteristic of $\F_q$. 
\end{proof}

The prime powers satisfying the condition of Theorem \ref{admissible} for a given pair $(v,k)$ will be said {\it admissible}.

\medskip
It is known that a cyclic $(v,k,1)$-design has necessarily $v\equiv 1$ or $k$ (mod $k(k-1))$,
whereas a 1-rotational $(v,k,1)$ design has necessarily $v\equiv k-1$ (mod $k(k-1))$.
In order to explain how to construct cyclic and 1-rotational Steiner 2-designs we need to recall the
notion of a difference family.

Given a set ${\cal F}=\{B_1,\dots,B_n\}$ of subsets of a group $G$, the {\it list of differences} of $\cal F$
is the multiset $\Delta{\cal F}$ consisting of all possible differences $x-y$ (if $G$ is additive) or all possible
ratios $xy^{-1}$ (if $G$ is multiplicative), with $(x,y)$ an ordered pair of distinct elements from the same $B_i$.

Given a subgroup $H$ of a group $G$, a $(G,H,k,\lambda)$ {\it difference family} is a set of $k$-subsets (called {\it base blocks})
of $G$ whose list of differences is precisely $\lambda$ times the set $G\setminus H$ (see \cite{B}). 

Speaking of a {\it cyclic $(G,k,1)$ difference family} we will mean a $(G,H,k,1)$ difference family where $G$ is a cyclic group
of order $v\equiv1$ or $k$ (mod $k(k-1))$ and $H$ is the subgroup of $G$ of order $1$ or $k$, respectively.

Speaking of a {\it $1$-rotational $(G,k,1)$ difference family} we will mean a $(G,H,k,1)$ difference family where $G$ is a cyclic group
of order $v\equiv k-1$ (mod $k(k-1))$ and $H$ is the subgroup of $G$ of order $k-1$.

The above terminology is justified by the following well-known result.
\begin{theorem}\label{designs via DFs} The following facts hold.
\item[(i)] Up to isomorphism, the cyclic $(v,k,1)$-designs with $v\equiv1$ (mod $k(k-1))$ are precisely those of the form $(G,{\cal B})$, where $G$ is a cyclic group of order $v$ and $\cal B$ is the set of all possible translates (under $G$) of the base blocks of a cyclic $(G,k,1)$ difference family.
\item[(ii)] Up to isomorphism, the cyclic $(v,k,1)$-designs with $v\equiv k$ (mod $k(k-1))$ are precisely 
those of the form $(G,{\cal B})$, where $G$ is a cyclic group of order $v$
and $\cal B$ is the set of all possible translates (under $G$) of the base blocks of a cyclic $(G,k,1)$ difference family and all possible cosets
of the subgroup of $G$ of order $k$.
\item[(iii)] Up to isomorphism, the $1$-rotational $(v,k,1)$-designs are precisely those of the form $(G\cup\{\infty\},{\cal B})$, where $G$ is a cyclic group of order $v-1$,
$\infty$ is a symbol not in $G$, and $\cal B$ is the set of all possible translates (under $G$) of the base blocks of a $1$-rotational $(G,k,1)$ difference 
family and all possible sets of the form $C\cup\{\infty\}$, with $C$ a coset of the subgroup of $G$ of order $k-1$.
\end{theorem}

From now on, $\F_q^*$ will denote the multiplicative group of $\F_q$.
Also, if $q\equiv 1$ (mod $n$), we denote by $R_{q,n}$ the subgroup of $\F_q^*$ of order $n$,
that is the group of $n$-th roots of unity in $\F_q$. 

The following theorems may be useful to find new additive Steiner 2-designs which are cyclic
or 1-rotational.

\begin{theorem}\label{CA}
Let $q=nv+1$ be a prime power and assume that
there exists a cyclic $(R_{q,v},k,1)$ difference family whose base blocks are all zero-sum
in $\F_q$. Then there exists an $\mathrm{EA}(q)$-additive $(v,k,1)$ design.
\end{theorem}
\begin{proof}
Let $\cal F$ be a difference family as in the assumption. Given that $R_{q,v}$ has order $v$, by the
definition of a cyclic difference family we necessarily have $v\equiv 1$ or $k$ (mod $k(k-1))$.

Consider first the case $v\equiv1$ (mod $k(k-1))$.
By Theorem \ref{designs via DFs}(i), any block of the cyclic $(v,k,1)$-design generated by $\cal F$ is of the form
$Bg$ with $B\in{\cal F}$ and $g\in R_{q,v}$. Then, considering that $B$ is zero-sum in $\F_q$ by assumption, this block
$Bg$ is zero-sum as well. 

Now consider the case $v\equiv k$ (mod $k(k-1))$.
By Theorem \ref{designs via DFs}(ii), any block of the cyclic $(v,k,1)$-design generated by $\cal F$ is either of the form
$Bg$ with $B\in{\cal F}$ and $g\in R_{q,v}$, or of the form $Hg$ with $H=R_{q,k}$ the subgroup of $R_{q,v}$ of order $k$ and $g\in R_{q,v}$. 
Both $Bg$ and $Hg$ are zero-sum in $\F_q$ since $B$ is zero-sum in $\F_q$ by assumption, and $H$ is also zero-sum since it is
known that any non-trivial subgroup of $\F_q^*$ is zero-sum. 

Thus, in any case, all the blocks of the design generated by $\cal F$ are zero-sum and the assertion follows.
\end{proof}

Similarly, we have the following.

\begin{theorem}\label{1rA}
Let $q=n(v-1)+1$ be a prime power and assume that
there exists a $1$-rotational $(R_{q,v-1},k,1)$ difference family whose base blocks are all zero-sum
in $\F_q$. Then there exists an $\mathrm{EA}(q)$-additive $(v,k,1)$-design.
\end{theorem}
\begin{proof}
Let $\cal F$ be a difference family as in the assumption. Given that $R_{q,v-1}$ has order $v-1$, by the
definition of a 1-rotational difference family we necessarily have $v\equiv k-1$ (mod $k(k-1))$.

Let us use $\cal F$ to construct a 1-rotational $(v,k,1)$-design as indicated in Theorem \ref{designs via DFs}(iii) 
by taking as extra point $\infty$ the zero of $\F_q$.
Thus any block of this design is of the form $Bg$ with $B\in{\cal F}$ and $g\in R_{q,v-1}$, or of the form $Hg\cup\{0\}$ with 
$H=R_{q,k-1}$ the subgroup of $R_{q,v-1}$ of order $k-1$ and $g\in R_{q,v-1}$. 
Reasoning as in Theorem \ref{CA}, we see that this block is zero-sum in $\F_q$ and the assertion follows.
\end{proof}

It is obvious that Theorem \ref{CA} (resp. Theorem \ref{1rA}) may be useful only
if we already know that a cyclic (resp. 1-rotational) $(v,k,1)$-design exists.
In the affirmative case we have to take the prime powers $q\equiv1$ (mod $v$) or (mod $v-1$),
respectively, which are admissible according to Theorem \ref{admissible}. 
Anyway it may happen that no admissible prime power satisfies the required congruence.
In this case we can exclude that the theorem leads to an additive $(v,k,1)$ design.
Here are two examples.
\begin{example}
It has been recently proved in \cite{Zhang} that there
exists a cyclic $(v,4,1)$-design for any admissible $v$ except 25 and 28. Thus, in particular,
there exists a cyclic $(100,4,1)$-design. 
Assume that there exists an $\mathrm{EA}(q)$-additive $(100,4,1)$-design for some prime power $q\equiv 1$ (mod 100). 
In this case, by Theorem \ref{admissible}, $q$ must be a power of a prime divisor of ${100-4\over4-1}=2^5$, hence $q=2^m$
for some $m$. On the other hand, it is obvious that there is no value of $m$ for which the congruence 
$2^m\equiv1$ (mod 100) holds.

We conclude that Theorem \ref{CA} cannot lead to any additive $(100,4,1)$-design. 
\end{example}

\begin{example}
It is known that a cyclic $(105,5,1)$-design exists \cite{AB}.
Assume that there exists an $\mathrm{EA}(q)$-additive $(105,5,1)$-design for some prime power $q\equiv 1$ (mod 105). 
In this case, by Theorem \ref{admissible}, $q$ must be a power of a prime divisor of ${105-5\over5-1}=5^2$, hence $q=5^m$
for some $m$. On the other hand, it is obvious that there is no value of $m$ for which the congruence 
$5^m\equiv1$ (mod 105) holds.

Also here, Theorem \ref{CA} cannot be useful to find an additive $(105,5,1)$-design. 
\end{example}

If, on the contrary, we have an admissible prime power $q=nv+1$ (resp. $q=n(v-1)+1$), 
we should make a computer search to establish whether a cyclic (resp. $1$-rotational) 
$(R_{q,v},k,1)$ (resp. $(R_{q,v-1},k,1)$) difference family exists.
Our experimental results seem to indicate that this search is usually heavy and that it
succeeds very rarely. This is expected especially when $q$ is much greater than $v$. Indeed, the larger is $q$ 
with respect to $v$, the lesser is the number 
of candidates for the base blocks of the required difference family (recall indeed that they have to be zero-sum!). 

Assume for instance that we want to construct an additive $(88,4,1)$-design using Theorem \ref{1rA}.
It makes sense since it is known that there exists a 1-rotational $(v,4,1)$-design for any $v\equiv4$ (mod 12)
except $v=28$ (see \cite{Zhao}).
Here we have ${v-k\over k-1}={88-4\over3}=28$ so that, by Theorem \ref{admissible}, the admissible $q$
must be powers of $2$ or $7$. Hence $q=2^m$ or $q=7^m$ some $m$. Recall, also, that Theorem \ref{1rA}
requires that $q\equiv 1$ (mod 87). The least $m$ for which $2^m\equiv1$ (mod 87) is $28$, whereas the least
$m$ for which $7^m\equiv1$ (mod 87) is 7. So, at this point, we should find a 
1-rotational $(R_{87,q},4,1)$ difference family whose base blocks are all zero-sum in $\F_q$
with either $q=2^{28}$ or $q=7^7$ which are both very large. In both cases, we have checked that $R_{87,q}$ does not even have
one zero-sum subset of size 4.

\section{An $\mathrm{EA}(2^8)$-additive 1-rotational and resolvable $(52,4,1)$-design}

Assume that we want to construct an additive $(52,4,1)$-design.
Given that 1-rotational $(52,4,1)$-designs exist,
it makes sense trying to use Theorem \ref{1rA}.
Here we have $v=52$ and $k=4$ so that ${v-k\over k-1}={48\over3}=16$. Hence the admissible values of $q$
must be powers of 2. The least power of 2 of the form $(v-1)n+1$ is $q=2^8$ which,
fortunately, is not so large. We have to find a 1-rotational $(R_{q,51},4,1)$ difference family
whose base blocks are all zero-sum in $\F_{q}$.

One can check that $p(x)=x^8+x^4+x^3+x^2+1$ is a primitive polynomial over $\Z_2$ so that
we have $\F_{q}=\Z_2[x]/(p(x))$ and $x$ is a primitive element of $\F_q$.
Given that $2^8=51\cdot 5+1$, the group $R_{q,51}$ can be viewed as the group of non-zero 
5th powers of $\F_q$. Hence  $g:=x^5$ is a generator of $R_{q,51}$ and it is easy to check that the desired
difference family is ${\cal F}=\{B_1,B_2,B_3,B_4\}$ with 

$$B_1=\{g, g^{12}, g^{16}, g^{39}\}; \quad B_2=\{g^3, g^4, g^{13}, g^{48}\};$$
$$B_3=\{g^6, g^{8}, g^{26}, g^{45}\};\quad B_4=\{g^7, g^{10}, g^{15}, g^{36}\}.$$

The above 1-rotational difference family $\cal F$ has the very special property to be {\it resolvable}
(see \cite{BZ}). It essentially means that the design generated by $\cal F$ is resolvable.
Its resolution is $\{g^h{\cal P}_0 \ | \ 0\leq h\leq 16\}$ with $${\cal P}_0=\{g^{17i}B_j \ | \ 0\leq i\leq 2; 1\leq j\leq 4 \} \ \cup \ \{B_0\}$$
where $B_0=\{0,1,g^{17},g^{34}\}$.
Thus we can state the following.
\begin{theorem} There exists an additive $(52,4,1)$-design that is both $1$-rotational and resolvable.
\end{theorem}
This design, together with the 
additive $(124,4,1)$-design constructed in \cite{BMN} is, up to now, the
only known handy example of a Steiner 2-design whose
parameters are not ``geometric", namely neither of the form $(q^n,q,1)$, nor of the form $({q^n-1\over q-1},{q^k-1\over q-1},1)$.

Yet, the above constructed design has something geometric anyway. 
Indeed we have discovered that it is isomorphic to the 1-rotational resolvable $(52,4,1)$-design in \cite{HTK} which is proved to be embeddable
in the desarguesian projective plane of order 16.

It may be interesting to know that we established their isomorphism
by checking that they both have {\it $2$-rank} 41 whereas the 2-rank of any other 1-rotational resolvable $(52,4,1)$ design is either 51 or 49.
We recall that there are precisely twenty-two 1-rotational resolvable $(52,4,1)$ designs; they have been
explicitly given in \cite{BZ2}.

\section{Three new $\mathrm{EA}(3^5)$-additive $(121,4,1)$-designs}

It is well-known that the point-line incidence structure associated with the projective space $\mathrm{PG}(n-1,q)$
is a $({q^n-1\over q-1},q+1,1)$-design. As recalled in the introduction, it has been proved that this design 
is $\mathrm{EA}(q^n)$-additive. 

On the other hand, it is possible that Theorem \ref{CA} may lead to at least two non-isomorphic
additive designs with parameters $({q^n-1\over q-1},q+1,1)$. We verified this for $q=3$ and $n=5$.

\begin{theorem} There exist at least four pairwise non-isomorphic $\mathrm{EA}(3^5)$-additive $(121,4,1)$-designs.
\end{theorem}
\begin{proof}
Given that $(121,4,1)$ are the parameters of the point-line design associated with $\mathrm{PG}(4,3)$,
we already know that there is at least one additive design with these parameters.
We now investigate whether other additive $(121,4,1)$-designs can be obtained via Theorem \ref{CA}.
Here we have $v=121$ and $k=4$ so that ${v-k\over k-1}={117\over3}=39$. Hence, the admissible values of $q$
are the powers of 3 or 13. The least power of 3 of the form $vn+1$ is $q=3^5$.

For applying Theorem \ref{CA} we have to find a cyclic $(R_{q,121},4,1)$ difference family
whose base blocks are all zero-sum in $\F_{q}$.
We are going to describe four of these difference families.

One can check that  $p(x)=x^5+2x+1$ is a primitive
polynomial over $\Z_3$ so that we can write $\F_q=\Z_3[x]/(p(x))$ and $x$ is a primitive element of $\F_q$.
Given that  $3^5=121\cdot 2+1$, the group $R_{q,121}$ can be viewed as the group of non-zero squares
of $\F_{121}$. Hence $g:=x^2$ is a generator of $R_{q,121}$.

Consider the 4-subsets of $R_{q,121}$
$$A_1=\{1,g,g^{5},g^{69}\},\quad A_2=\{1,g,g^{21},g^{55}\},$$
$$A_3=\{1,g,g^{52},g^{93}\},\quad A_4=\{1,g,g^{65},g^{78}\},$$
$$B_1=\{1,g^2,g^{46},g^{74}\},\quad B_2=\{1,g^4,g^{79},g^{95}\},$$
$$B_3=\{1,g^4,g^{15},g^{78}\},\quad B_4=\{1,g^2,g^{25},g^{116}\},$$

and set, for $1\leq i\leq 4$, 
$${\cal F}_i=\{g^{3^j}A_i, g^{3^j}B_i \ | \ 0\leq j\leq 4\}.$$

One can check that each ${\cal F}_i$ is a cyclic $(R_{q,121},4,1)$ difference family
whose base blocks are zero-sum in $\F_q$.
Using GAP \cite{Gap} we also checked that the designs generated by them are
pairwise non-isomorphic. In particular, the design generated by ${\cal F}_1$ is
the point-line design associated with $\mathrm{PG}(4,3)$.
\end{proof}

\section{On the possible existence of an additive $(511,7,1)$-design}

As a special case of Theorem \ref{CA} we have the following.
\begin{corollary}\label{2analogalike}
If there exists a cyclic $(\F_{2^v}^*,k,1)$ difference family whose base blocks are all zero-sum
in $\F_{2^v}$, then there exists an $\mathrm{EA}(2^v)$-additive $(2^v-1,k,1)$-design.
\end{corollary}

By interpreting $\F_{2^v}$ as the $v$-dimensional vector space over $\Z_2$, 
 the set of points of $\mathrm{PG}(v-1,2)$ is $\F_{2^v}^*$, whereas the $(k-1)$-dimensional
subspaces of $\mathrm{PG}(v-1,2)$ are the $(2^k-1)$-subsets of $\F_{2^v}^*$ 
which, together with the zero vector, form a $k$-dimensional subspace of $\F_{2^v}$.
We have the following remarkable variant of Corollary \ref{2analogalike}, which can be deduced
from Theorem 6.4 in \cite{BNW}. We prove it anyway along the lines of Theorem \ref{CA}
for the convenience of the reader.

\begin{theorem}\label{2analog}
If there exists a cyclic $(\F_{2^v}^*,2^k-1,1)$ difference family whose base blocks are all 
$(k-1)$-dimensional subspaces of $\mathrm{PG}(v-1,2)$, then there exists the $2$-analog of a $(v,k,1)$-design.
\end{theorem}
\begin{proof}
Let $\cal F$ be a difference family as in the assumption. 
By Theorem \ref{designs via DFs}, any block of the cyclic $(2^v-1,2^k-1,1)$-design generated by $\cal F$ is of the form
$Bg$ with $B\in{\cal F}$ and $g\in \F_{2^v}^*$ or, possibly, a coset of the subgroup of $\F_{2^v}^*$ of order 
$2^k-1$ in the case that $k$ is a divisor of $v$. 
The fact that $B$ is a $(k-1)$-dimensional subspace of $\mathrm{PG}(v-1,2)$
easily implies that the block $Bg$ has the same property. Also note that when $k$ divides $v$
every coset of the subgroup of $\F_{2^v}^*$ of order $2^k-1$ is also a 
$(k-1)$-dimensional subspace of $\mathrm{PG}(v-1,2)$.

Thus, all the blocks of the design generated by $\cal F$ are $(k-1)$-dimensional subspaces of $\mathrm{PG}(v-1,2)$ and the assertion follows.
\end{proof}

The parameters of the additive designs with less than 100,000 points which are potentially obtainable 
from Corollary \ref{2analogalike} are only the following:
$$(127,7,1), \ (511,6,1), \ (511,7,1), \ (8191,6,1), \ (8191,7,1), \ (8191,10,1).$$
Among them, the 2-analogs which are potentially obtainable 
from Theorem \ref{2analog} have parameters
$$(127,7,1), \quad (511,7,1), \quad (8191,7,1).$$
They would be 2-analogs of a $(v,3,1)$ design with $v=7, 9$ and $13$, respectively.

As a matter of fact the existence of a $(127,7,1)$-design is not known even without any special requisite. 
Also, it has been proved that the automorphism group of a putative 2-analog of a $(7,3,1)$-design has 
order at most 2 (see ~\cite{BKN15, KKW16}) so that Theorem \ref{2analog} cannot lead to such a 2-analog.
The 2-analog of a $(13,3,1)$-design has been constructed in \cite{BEOVW}
with the implicit use of Theorem \ref{2analog}.

Given that cyclic $(511,7,1)$-designs exist (see Theorem 3.1 in \cite{B2}), it makes sense 
to try using Theorem \ref{CA}. For now, however, an exhaustive computer search with the
only use of this theorem seems to be unfeasible.

Unfortunately, Theorem \ref{2analog} does not lead to 
any 2-analog of a $(9,3,1)$-design. Hence, we obtain the following result.
\begin{theorem}\label{(9,3,1)_2}
A putative $2$-analog of a $(9,3,1)$-design cannot be cyclic.
\end{theorem}
This will be proved by means of two different approaches in the next subsections. 

\subsection{Proof of Theorem \ref{(9,3,1)_2} via Kramer-Mesner method}\label{KM}

A straightforward approach to the construction of a $(v,k,1)$-design $(V,\mathcal{B})$
is to consider the matrix
\[
M_{k}^v := (m_{i,j}),\quad i=1,\ldots,{v\choose 2},\;
j=1,\ldots,{v\choose k},
\]
where the rows of $M_{k}^v$ are indexed by the set $V\choose2$ of all $2$-subsets of $V$ and the columns
by the set $V\choose k$ of all $k$-subsets of $V$. If the $i$-th $2$-subset is
contained in the $j$-th $k$-subset, $m_{i,j}=1$, otherwise $m_{i,j}=0$.
Then, a $(v,k,1)$-design corresponds to $\{0,1\}$-solutions $x$
of the system of linear equations
\[
M_{k}^v \cdot x = \begin{pmatrix}1\\ \vdots \\1\end{pmatrix}
\]
over the rationals.
In computer science, the above linear system 
is known as \emph{exact cover problem}, 
see \cite{knuth:fasc5}, and is well-known to be NP-hard.

For the most interesting parameters
$v,k$, the size of the matrix $M_{k}^v$ will be prohibitively large and 
even the fastest computers available are not able to solve 
these systems of ${v\choose 2}$ linear equations.

However, by assuming a group action on $V$, sometimes the size of $M_{k}^v$ can be
dramatically reduced. 
A group $G$ acting on $V$ induces also an action on $V\choose2$ and $V\choose k$.
Now, $M^G_{k}=(m_{i,j})$ denotes the matrix where $m_{ij}$ counts 
how often a representative of the $i$-th orbit on $V\choose2$
is contained in the $j$-th orbit on $V\choose k$.
Kramer and Mesner observed the following.
\begin{theorem}\label{KM76}(\cite{KramerMesner:76})
	A $(v,k,1)$-design with
	an automorphism group $G$ exists if and only if there is a $\{0,1\}$-solution $x$ to the
	linear system of equations
	\[
	M^G_{k}\cdot x = \begin{pmatrix}1\\ \vdots \\1\end{pmatrix}
	\]
	over the rationals.
\end{theorem}
The matrix $M^G_{k}$ is now called Kramer-Mesner matrix, since 
it was first introduced formally by Kramer and Mesner in \cite{KramerMesner:76}, 
albeit this approach had been used implicitly by others before.

The same approach can be applied to search for $q$-analogs of
$(v,k,1)$-designs. 
That is, we search for a $q$-analog of a $(v,k,1)$-design having a 
given group $G\leq \mbox{PGL}(v, q)$ as a group of automorphisms.
Instead of orbits on $2$-subsets and $k$-subsets, 
$G$-orbits on $2$-dimensional subspaces
and $k$-dimensional subspaces have to be considered, see \cite{qdesignscomputer2017,qdesigns2017}
for the full details and results.

Analog to classic designs, 
the set of blocks of a $q$-analog of a $(v,k,1)$-design with an automorphism group $G\leq \mbox{PGL}(v, q)$,
would be the union of orbits of $k$-dimensional subspaces 
under the action of $G$.
Now, the Kramer-Mesner matrix $M_{k}^G$ carries the information, how many elements of the orbit $K^G$ 
for a given $k$-dimensional  subspace $K\leq \F_{q}^v$ contain a given $2$-dimensional  subspace $T\leq \F_{q}^v$.

The $q$-analog version of Theorem~\ref{KM76} by Kramer and Mesner is that a 
$q$-analog of a $(v,k,1)$-design with an automorphism group $G$ exists if and only if
there is a $\{0,1\}$-solution $x$ of the system of linear equations
\[
M_{k}^G \cdot x = \begin{pmatrix}1\\ \vdots\\ 1\end{pmatrix}
\]
over the rationals.

In the special case of a \emph{cyclic} $2$-analog of a $(9,3,1)$-design, the set of blocks 
would be the union of orbits of $3$-dimensional subspaces of $\F_2^9$
under the action of the cyclic group $G = \F_{2^9}^*$ (a Singer cycle) of order $511$.

The action of $G$ partitions the $43\,435$ $2$-dimensional subspaces into $85$ orbits, all of size $511$.
Similarly, the $788\,035$ $3$-dimensional subspaces are partitioned by $G$ into $1543$ orbits,
where one orbit is of size $73$ and all the others are of size $511$.

Among the $1543$ orbits on $3$-dimensional subspaces, $1459$ might be used for the construction of a putative 
$2$-analog of a $(9,3,1)$ design. The other orbits are incompatible and can be ignored
because at least one of the $2$-dimensional subspaces is already contained in more than one of their orbit members, 
i.e., the corresponding column of the Kramer-Mesner matrix would have an entry larger than one.

In a straightforward Python implementation of the last author, 
the construction of the Kramer-Mesner matrix $M_{3}^G$ for a putative $2$-analog of a $(9,3,1)$-design 
takes 33 seconds (using \texttt{pypy3}).
The exhaustive enumeration of all solutions of the resulting linear system needs approximately
20 seconds using either one of the
exact cover algorithms by Knuth \cite{knuth:fasc5,knuth:fasc7}, implemented by him in C. 
The computation shows that a putative $2$-analog of a $(9,3,1)$-design cannot be cyclic.

The computations were performed on an Intel(R) Xeon(R) E-2288G CPU (3.70GHz) running Linux, Debian 13.

We do not claim that we are the first to show with the Kramer-Mesner approach that a 
$2$-analog of a $(9,3,1)$-design cannot be cyclic. The same group and design parameters 
had already been tested by the authors of \cite{BEOVW} after their construction of a 
$2$-analog of a $(13,3,1)$-design.
However, apparently this negative result was never published.

\subsection{Proof of Theorem \ref{(9,3,1)_2}  via some geometric arguments}\label{AM}

Let $\mathcal{D}=(V,\mathcal{B})$ be the $2$-analog of the point-line design of $\mathrm{AG}(2,3)$. Then $V$ coincides with the point-set of $\mathrm{PG}(8,2)$ and $\mathcal{B}$ consists of some suitable copies of $\mathrm{PG}(2,2)$. Here, we assume that $%
\mathcal{D}$ admits $G\cong \mathbb{Z}_{511}$ as a point-regular automorphism
group. Therefore, $G$ is a Singer subgroup of $\mathrm{PGL}(9,2)$, and hence $G$
preserves a unique Desarguesian $2$-spread $\Sigma $ of $\mathrm{PG}(8,2)$, namely
a set of $73$ pairwise disjoint planes of $\mathrm{PG}(8,2)$, by \cite[Lemma 2.1]%
{Dru}. Thus, $G$ preserves a copy of the plane $\mathrm{PG}(2,8)$, where the
points are the elements of $\Sigma $ and the lines are the $5$-dimensional subspaces of $\mathrm{PG}(8,2)$ generated by any two elements of $\Sigma$. In particular, $G$
acts on $\mathrm{PG}(2,8)$ in a natural way inducing $W\cong \mathbb{Z}_{73}$ on it.

Let $B$ be any block of $\mathcal{D}$ such that $B\notin \Sigma $, then $%
\left\vert B\cap X\right\vert \leq 1$ for any $X\in \Sigma $ since $\lambda =1$ and $X$ is also a
block of $\mathcal{D}$. Therefore, the incidence
structure $\pi _{B}$, whose point-set consists of the $7$ elements of 
$\Sigma$ intersecting $B$ in a point and whose block-set consists of the
intersections of $B$ with any $5$-dimensional subspace of $\mathrm{PG}(8,2)$
generated by two elements of $\Sigma$, is an isomorphic copy of $\mathrm{PG}(2,2)$ naturally embedded in $\mathrm{PG}(2,8)$.

\begin{lemma}
	\label{LemJedan}Let $\mathcal{D}^{\prime }=(V^{\prime },\mathcal{B}^{\prime })$, where $V^{\prime }$
	is the point-set of $\mathrm{PG}(2,8)$ and $\mathcal{B}^{\prime }=\left\{ \pi
	_{B} \mid B\in \mathcal{B}\setminus \Sigma \right\} $. Then the following hold:
	
	\begin{enumerate}
		\item $\mathcal{D}^{\prime }=(V^{\prime },\mathcal{B}^{\prime })$
		is a $2$-$(73,7,7)$ design admitting $W$ as a point-regular automorphism
		group;
		
		\item $\mathcal{B}^{\prime }$ consists of $12$ orbits under $W$.
	\end{enumerate}
\end{lemma}

\begin{proof}
	It follows from its definition that $\mathcal{D}^\prime$ has $v^\prime = 73$ points, and each block of $\mathcal{D}^\prime$ has $k^\prime = 7$ points. Clearly, $\pi_{B^\prime} = \pi_{B}$ for each $B^\prime \in B^{C}$, where $C$ is the subgroup of order $7$ of $G$. Conversely, suppose that $\pi_{B^\prime} = \pi_{B}$. Let $X_i \in \Sigma$, $i = 1, \ldots, 7$, be such that $\lvert B \cap X_i \rvert = 1$. Then $\lvert B^\prime \cap X_i \rvert = 1$ for each $i = 1, \ldots, 7$. Hence, $\left\{X_1, X_2, \ldots, X_7\right\}$ and $B^{C}$ are the two systems of a Segre variety $\mathcal{S}_{2,2}$ of $\mathrm{PG}(8,2)$, by \cite[Section 25.5]{HT}. Therefore, $B^\prime \in B^{C}$ by \cite[Theorem 25.5.11]{HT}, since $\pi{B^\prime} = \pi_{B}$. Thus, $\pi_{B^\prime} = \pi_{B}$ if and only if $B^\prime \in B^{C}$.
	
	Let $X, Y$ be any two distinct points of $\mathcal{D}^\prime$. Then they correspond to two disjoint copies of $\mathrm{PG}(2,2)$ inside $\mathrm{PG}(8,2)$. Now, for any $x \in X$ and $y \in Y$, there exists a unique block $B$ of $\mathcal{D}$ containing them, since $\mathcal{D}$ is a $2$-design with $\lambda = 1$. Moreover, $B \cap X = \{x\}$ and $B \cap Y = \{y\}$. Therefore, the number of blocks intersecting $X$ and $Y$ in a point is $49$. 
	
	On the other hand, each element of $(B^\prime)^{C}$ intersects $X$ and $Y$ in a point, and these correspond to the same element $\pi_{B^\prime}$ of $\mathcal{B}^\prime$. Thus, the number of blocks of $\mathcal{D}^\prime$ intersecting $X$ and $Y$ is precisely $7$. Hence, $\mathcal{D}^\prime$ is a $2$-$(73,7,7)$ design.
	
	Clearly, $W$ is a point-regular automorphism group of $\mathcal{D}^\prime$, and we obtain (1). Furthermore, $W \cong \mathbb{Z}_{73}$ acts semiregularly on $\mathcal{B}^\prime$. Hence, (2) holds true since 
	\[
	\lvert \mathcal{B}^\prime \rvert = \frac{\lvert \mathcal{B} \rvert - \lvert \Sigma \rvert}{7}
	= \frac{6205 - 73}{7}
	= 876 = 12 \cdot 73.
	\]
\end{proof}

\begin{lemma}
	\label{digest}If $\pi ,\pi ^{\prime }\in \mathcal{B}^{\prime }$ are such $%
	\pi \cap \ell =\pi ^{\prime }\cap \ell $ with $\lvert \pi \cap \ell
	\rvert =3$, then $\pi =\pi ^{\prime }$.
\end{lemma}

\begin{proof}
	
	Let $\pi , \pi ^{\prime } \in \mathcal{B}^{\prime }$ be such that $\pi \cap \ell = \pi ^{\prime } \cap \ell$ with $\lvert \pi \cap \ell \rvert = 3$. Then there exist two blocks of $\mathcal{D}$, say $B = \left\langle e_{1}, e_{2}, e_{3} \right\rangle^{\ast}$ and $B^{\prime} = \left\langle e_{1}^{\prime}, e_{2}^{\prime}, e_{3}^{\prime} \right\rangle^{\ast}$, with 
	\[
	B \cap \left\langle e_{i} \right\rangle_{\mathbb{F}_{8}}^{\ast} = \left\langle e_{i} \right\rangle_{\mathbb{F}_{2}}^{\ast}, \quad 
	B \cap \left\langle e_{1} + e_{2} \right\rangle_{\mathbb{F}_{8}}^{\ast} = \left\langle e_{1} + e_{2} \right\rangle_{\mathbb{F}_{2}}^{\ast},
	\]
	and 
	\[
	B^{\prime} \cap \left\langle e_{i} \right\rangle_{\mathbb{F}_{8}}^{\ast} = \left\langle e_{i}^{\prime} \right\rangle_{\mathbb{F}_{2}}^{\ast}, \quad 
	B^{\prime} \cap \left\langle e_{1} + e_{2} \right\rangle_{\mathbb{F}_{8}}^{\ast} = \left\langle e_{1}^{\prime} + e_{2}^{\prime} \right\rangle_{\mathbb{F}_{2}}^{\ast}.
	\]
	Then there exist integers $i, j, h$ such that 
	\[
	e_{1}^{\prime} = e_{1}^{\varphi^{i}}, \quad 
	e_{2}^{\prime} = e_{2}^{\varphi^{j}}, \quad 
	e_{1}^{\prime} + e_{2}^{\prime} = (e_{1} + e_{2})^{\varphi^{h}},
	\]
	where $\langle \varphi \rangle = C$, since $\langle e_{1} \rangle_{\mathbb{F}_{8}}^{\ast}$, $\langle e_{2} \rangle_{\mathbb{F}_{8}}^{\ast}$, and $\langle e_{1} + e_{2} \rangle_{\mathbb{F}_{8}}^{\ast}$ are distinct $C$-orbits. 
	
	Then $e_{1}^{\varphi^{i}} + e_{2}^{\varphi^{j}} = e_{1}^{\varphi^{h}} + e_{2}^{\varphi^{h}}$, since $\varphi$ is linear, and hence 
	\[
	e_{1}^{\varphi^{i - h}} + e_{1} = e_{2}^{\varphi^{j - h}} + e_{2},
	\]
	which forces $\langle e_{1} \rangle_{\mathbb{F}_{8}}^{\ast} = \langle e_{2} \rangle_{\mathbb{F}_{8}}^{\ast}$, whereas $\langle e_{1} \rangle_{\mathbb{F}_{8}}^{\ast}$ and $\langle e_{2} \rangle_{\mathbb{F}_{8}}^{\ast}$ are distinct $C$-orbits. Therefore $h = i = j$, and hence $\{ e_{1}^{\prime}, e_{2}^{\prime} \} \subset B^{\prime} \cap B^{\varphi^{h}}$. 
	
	Then $B^{\varphi^{h}} = B^{\prime}$, since $\mathcal{D}$ is a $2$-design with $\lambda = 1$, and so $\pi = \pi^{\prime}$.

\end{proof}

\begin{definition}
	Let $\ell $ be any line of $\mathrm{PG}(2,8)$ and $\mathcal{O}$ be any block-$W$-orbit of $\mathcal{D}^{\prime }$. The set
	\[
	I(\mathcal{O},\ell) = \left\{ \pi \cap \ell \mid \lvert \pi \cap \ell \rvert = 3, \, \pi \in \mathcal{O} \right\}
	\] 
 	is called the \emph{imprint of }$\mathcal{O}$ on $\ell $.
\end{definition}

\begin{corollary}
	\label{CorJedan}Let $\ell $ be any line of $\mathrm{PG}(2,8)$ and $\mathcal{O}$ be
	any block-$W$-orbit of $\mathcal{D}^{\prime }$. Then $\lvert I(\mathcal{O},\ell )\rvert =7$.
\end{corollary}

\begin{proof}
	The incidence structure $(\mathcal{O}, \mathcal{B})$, with incidence relation given by the pairs $(\pi, m) \in \mathcal{O} \times \mathcal{B}$ such that $\lvert \pi \cap m \rvert = 3$, is a symmetric $1$-design by \cite[1.2.6]{Demb}, since both $\mathcal{O}$ and $\mathcal{B}$ are $W$-orbits of size $73$. Hence,
	\[
	\left\lvert \left\{ \pi \mid \lvert \pi \cap \ell \rvert = 3, \ \pi \in \mathcal{O} \right\} \right\rvert = 7,
	\]
	since the number of lines of $\mathrm{PG}(2,8)$ secant to $\pi$ is $7$. 
	
	If $\lvert I(\mathcal{O}, \ell) \rvert < 7$, then there exist $\pi_{1}, \pi_{2} \in I(\mathcal{O}, \ell)$ such that $\pi_{1} \cap \ell = \pi_{2} \cap \ell$ with $\lvert \pi_{1} \cap \ell \rvert = 3$, and hence $\pi_{1} = \pi_{2}$ by Lemma \ref{digest}. Therefore,
	\[
	7 \leq \lvert I(\mathcal{O}, \ell) \rvert \leq \left\lvert \left\{ \pi \mid \lvert \pi \cap \ell \rvert = 3, \ \pi \in \mathcal{O} \right\} \right\rvert = 7,
	\]
	and hence $\lvert I(\mathcal{O}, \ell) \rvert = 7$.
\end{proof}

\begin{definition}
	Let $\ell$ be any line of $\mathrm{PG}(2,8)$. A family $\mathcal{F}$ of copies of $\mathrm{PG}(2,2)$ provides a \emph{perfect cover of} $\ell$ if 
	\[
	\left\{ \pi \cap \ell \mid \lvert \pi \cap \ell \rvert = 3, \, \pi \in \mathcal{F} \right\}
	\]
	is the set of all sublines $\mathrm{PG}(1,2)$ of $\ell$.
\end{definition}

\begin{theorem}
	\label{TeoJedan}Let $\mathcal{O}_{j}$, $j=1,\ldots,12$, be the $W$-orbits on $\mathcal{B}^{\prime }$. For each line $\ell $ of $\mathrm{PG}(2,8)$, the set  
	\[
	\bigcup_{j=1}^{12}I(\mathcal{O}_{j},\ell )
	\]%
	provides a perfect cover of $\ell $.
\end{theorem}

\begin{proof}
	It follows from Corollary \ref{CorJedan} that $\lvert I(\mathcal{O}_{j},\ell) \rvert = 7$ for each $j = 1, \dots, 12$. Therefore,
	\[
	\left\lvert \bigcup_{j=1}^{12} I(\mathcal{O}_{j},\ell) \right\rvert \leq 84.
	\]
	
	If $I(\mathcal{O}_{j_{1}},\ell) \cap I(\mathcal{O}_{j_{2}},\ell) \neq \emptyset$ for some distinct $1 \leq j_{1}, j_{2} \leq 12$, then there exist $\pi_{1} \in \mathcal{O}_{j_{1}}$ and $\pi_{2} \in \mathcal{O}_{j_{2}}$ with $\lvert \pi_{1} \cap \ell \rvert = 3$ such that $\pi_{1} \cap \ell = \pi_{2} \cap \ell$. Hence, $\pi_{1} = \pi_{2}$ by Lemma \ref{digest}, so $\mathcal{O}_{j_{1}} = \mathcal{O}_{j_{2}}$, which is a contradiction since these are distinct $W$-orbits. 
	
	Thus, the sets $I(\mathcal{O}_{j},\ell)$ are pairwise disjoint, and hence
	\[
	\left\lvert \bigcup_{j=1}^{12} I(\mathcal{O}_{j},\ell) \right\rvert = 84.
	\]
	
	On the other hand, the set of all sublines $\mathrm{PG}(1,2)$ of $\ell$ has size $\lvert \mathrm{PSL}_{2}(8) : \mathrm{PSL}_{2}(2) \rvert = 84$, and hence $\bigcup_{j=1}^{12} I(\mathcal{O}_{j},\ell)$ provides a perfect cover of $\ell$.
\end{proof}

We implemented an algorithm that uses Theorem \ref{TeoJedan} to show
the non-existence of $\mathcal{D}^{\prime }$ and hence it provides a proof of Theorem \ref{(9,3,1)_2}. We operate as follows.

\subsubsection{Description of the algorithm}

This algorithm tests whether there exist twelve distinct orbits
$\mathcal{O}_i$ of blocks in $\mathrm{PG}(2,8)$ under the action of the group
$W \cong \mathbb{Z}_{73}$ such that their imprints $I(\mathcal{O}_i,\ell)$ on a fixed line
$\ell$ are pairwise disjoint.

The procedure has two phases. Phase~A constructs  $n\le 7$ such
orbits; Phase~B searches for the remaining $12-n$, if they exist. The entire algorithm has been implemented and executed using a Python program.

$\bullet$ \textbf{Phase~A:} We first fix the line
$\ell=\{1,2,35,37,42,45,47,54,63\}$ and its seven $3$-point \emph{sublines}
$s_i$ of the form $\{1,2,x\}$, for $i=1,\dots,7$.
All orbits are precomputed (via GAP \cite{Gap}) and loaded from disk. For each orbit $\mathcal{O}$ we
compute its imprint on $\ell$, thereby forming a list
$L$ of pairs $(\mathcal{O}, I(\mathcal{O},\ell))$.
We then define $7$ new lists as follows:
\begin{equation*}
	L_i=\bigl\{\,\big(\mathcal{O}, I (\mathcal{O},\ell) \big)\in L \ \bigm|\ s_i\in I(\mathcal{O},\ell)\,\bigr\},
	\qquad i=1,\dots,7.
\end{equation*}
Each $L_i$ has $105$ elements (clearly, the $L_i$'s are not mutually disjoint).
We iterate over $L_1$ with an outer for-loop, fixing an element
$(\mathcal{O}_1,I_1)\in L_1$. The goal is to filter the remaining sets
$L_2,\dots,L_7$ to those orbits whose imprints are disjoint from $I_1$.
Proceeding in the fixed order $L_2,L_3,\ldots$ may prematurely eliminate
feasible solutions: a priori we do not know whether a supposed solution present,
for example, $s_1$ and $s_2$ in the imprint of the same  orbit. If this happen, then
$L_2$ filtered by disjointness from $I_1$ will be empty and the loop would
stop without further exploring this option.

Therefore, after fixing $(\mathcal{O}_1,I_1)$ we filter \emph{each} of
$L_2,\dots,L_7$, producing six lists
\begin{equation*}
	\bigl\{\,\big(\mathcal{O}, I (\mathcal{O},\ell) \big)\in L_i \ \bigm|\ I(\mathcal{O},\ell)\cap I_1=\emptyset\,\bigr\},
	\qquad i=2,\dots,7.
\end{equation*}
Among these we select the one of maximum size and call it $L_2'$; we denote by
$i_2$ the index from which it originates (i.e., $L_2'$ is $L_{i_2}$ filtered).
If two different lists have maximum size, then we fix the first one encountered. 

If $L_2'\neq\emptyset$, we iterate over
$(\mathcal{O}_2,I_2)\in L_2'$. Next, for the remaining five lists with
$i\in\{2,\dots,7\}\setminus\{i_2\}$ we filter by disjointness from both $I_1$
and $I_2$,
\begin{equation*}
	\bigl\{\,\big(\mathcal{O}, I (\mathcal{O},\ell) \big)\in L_i \ \bigm|\ I(\mathcal{O},\ell)\cap I_j=\emptyset,\ j=1,2\,\bigr\},
\end{equation*}
and select the largest filtered list $L_3'$ with its index $i_3$. 

If $L_3'\neq\emptyset$, we fix $(\mathcal{O}_3,I_3)\in L_3'$ and repeat. 

At step $k$, we iterate 
$(\mathcal{O}_k,I_k)\in L_k'$. For the  remaining $7-k$ lists with
$i\in\{2,\dots,7\}\setminus\{i_2,i_3,\dots,i_k\}$, we filter by disjointness from all $I_1$, $I_2$,\dots, $I_k$, creating the following new lists:
\begin{equation*}
	\bigl\{\,\big(\mathcal{O}, I (\mathcal{O},\ell) \big)\in L_i \ \bigm|\ I(\mathcal{O},\ell)\cap I_j=\emptyset,\ j=1,2,\dots,k\,\bigr\},
\end{equation*}
We select the largest filtered list $L_{k+1}'$ with its index $i_{k+1}$. If
$L_{k+1}'\neq\emptyset$, we fix $(\mathcal{O}_{k+1},I_{k+1})\in L_{k+1}'$ and repeat.

Two outcomes are possible: either we reach $k=7$ or the current maximal
filtered list $L_n'$ is empty (and hence all other filtered lists are empty,
since $L_n'$ has maximal cardinality). In the first case, each $s_i$ lies in a
different imprint $I_j$. In the second case, at least two sublines $s_i,s_j$
appear in the same imprint. In both cases, the sublines $s_1,\dots,s_7$ have
been covered as required.

$\bullet$ \textbf{Phase~B:} We then call a subroutine that attempts to build the remaining $12-n$ orbits.
Its input is a partial family
$\{(\mathcal{O}_1,I_1),\dots,(\mathcal{O}_n,I_n)\}$ of already constructed orbits and
their imprints.

The subroutine proceeds as follows. First, it forms
$L_{\text{partial}}\subseteq L$, consisting of all pairs $(\mathcal{O},I(\mathcal{O},\ell))$ whose
imprints are disjoint from every $I_i$, $1\le i\le n$.
If $\lvert L_{\text{partial}}\rvert$ is strictly smaller than the number of
orbits still required to reach $12$, the subroutine terminates immediately
(pruning the branch).

Otherwise, it performs a depth-first search using an explicit stack. Each stack
level corresponds to selecting a new orbit, so at level $k$ the algorithm has
chosen $n+k+1$ orbits (the search starts at level $0$). At each step:
\begin{enumerate}
	\item the current candidate list is scanned sequentially;
	\item if the next candidate $(\mathcal{O},I(\mathcal{O},\ell))$ is \emph{valid} (namely, the
	filtered list obtained by enforcing disjointness has cardinality
	greater than or equal to the number of orbits still required to reach
	$12$) it is pushed on the stack together with its updated disjoint
	candidate list;
	\item if no viable candidates remain, the algorithm backtracks to the
	previous level and continues with the next unexplored element.
\end{enumerate}

If the depth reaches $11-n$, a complete configuration of twelve orbits is
found; the solution is written to file and the program terminates. This
subroutine guarantees exhaustive exploration, while the pruning rule discards
unpromising branches as early as possible. The explicit stack avoids deep
recursion and keeps the search efficient and memory-safe.

If, for every partial family $\mathcal{O}=\{(\mathcal{O}_1,I_1),\dots,(\mathcal{O}_n,I_n)\}$
generated by Phase~A, Phase~B never reaches twelve orbits, the program ends and
reports that no solution was found.

\subsection{Computational platform and runtime.}
All computations were performed on a MacBook Air (13.6-inch, 2022) equipped with an Apple M2 processor (8-core CPU, 8-core GPU), 8 GB of unified memory, and a 512 GB SSD, running macOS Sequoia.  

The algorithm completed the exhaustive search in approximately 7-8 hours of time. No complete solution was found within this runtime, which provides a proof for the non-existence of the desired design.

\section*{Acknowledgments}
The first, second, and third author thank the Italian National Group for Algebraic and Geometric Structures and their Applications (GNSAGA-INdAM) for its support to their research.

\end{document}